\documentclass{amsart} 
\usepackage{amssymb,latexsym,amscd,hyperref}
\usepackage{prelim2e}   %vorl"aufige Version

\theoremstyle{plain}
\newtheorem{theorem}{Theorem}

\newtheorem{lemma}{Lemma}
\newtheorem{remark}{Remark}
\newtheorem{remarks}{Remarks}

\newtheorem{algorithm}{Algorithm}

\theoremstyle{definition}

\theoremstyle{remark}

\newcommand {\Algo}[5]
            {{\goodbreak\topsep0pt\partopsep0pt
            \begin{algorithm} \label{#1}
                {\em (#2)} %\vspace{-9pt}
                \begin{list}{}{\leftmargin2.0cm\def\makelabel##1{{\em \underline{##1}:\hfill}}
                \labelwidth1.8cm\labelsep.2cm\itemsep.2cm} 
\item[Input] #3
\item[Output] #4
\vspace{6pt}
#5 \end{list}
\end{algorithm}
            \goodbreak}}
 
\newcommand {\Q}{{\mathbb{Q}}}

\newcommand {\Z}{{\mathbb{Z}}}

\newcommand {\F}{{\mathbb{F}}}

\newcommand {\OO}{{\mathcal{O}}}

\newcommand {\idp}{{\mathfrak{p}}}

\newcommand{\Rank}      {\mathop{\rm {Rank}}}

\newcommand{\Cl}      {\mathop{\rm {Cl}}}
\newcommand{\rk}      {\mathop{\rm {rk}}}

\newcommand{\jac}[2] {{\left(\frac{#1}{#2}\right)}}

\newcommand{\eps}{\varepsilon}
\newcommand{\exponent}{{c}}

\begin{document}

\title[Imaginary quadratic number fields with small exponent]{Imaginary quadratic number fields with class groups of small exponent}
\author{Andreas-Stephan Elsenhans}
\address{Universit\"at Paderborn,
Fakult\"at EIM,
Institut f\"ur Mathematik,
Warburger Str.~100,
33098~Paderborn,
Deutschland }
\email{elsenhan@math.uni-paderborn.de}
%\ead[url]{http://math.uni-paderborn.de/ag/ca/elsenhans/}

\author{J\"urgen Kl\"uners}
\address{Universit\"at Paderborn,
Fakult\"at EIM,
Institut f\"ur Mathematik,
Warburger Str.~100,
33098~Paderborn,
Deutschland }
\email{klueners@math.uni-paderborn.de}
%\ead[url]{http://math.uni-paderborn.de/ag/klueners/}
\author{Florin Nicolae}
\address{``Simion Stoilow'' Institute of Mathematics of the
Romanian Academy, P.O.BOX 1-764,RO-014700 Bucharest, Romania}
\email{florin.nicolae@imar.ro}

\begin{abstract}
  Let $D<0$ be a fundamental discriminant and denote by $E(D)$
  the exponent of the ideal class group $\Cl(D)$ of $K=\Q(\sqrt{D})$. Under the assumption that no Siegel zeros exist we compute all such $D$ with $E(D)$ is a divisor of $8$. We compute
  all $D$ with $|D|\leq 3.1\cdot 10^{20}$ such that $E(D)\leq 8$.
  
%{\it Key words:} Number fields, Galois extension, Artin L-function

%{\it Mathematics Subject Classification:} 11R42
\end{abstract}

\maketitle

\section{Introduction}

Let $D$ be a fundamental discriminant, i.e. the discriminant of
a quadratic number field. For $D<0$ let $E(D)$ be the exponent of the ideal 
class group $\Cl(D)$ of the imaginary quadratic field $K=\Q(\sqrt{D})$. 
Under the Extended Riemann Hypothesis it is known (see \cite{BoKi}, \cite{We}) 
that $E(D)\gg \frac{\log |D|}{\log\log |D|}$. Without any unproved hypothesis 
it is not even known that $E(D)\to\infty$. In \cite{We}, Theorem 1, it is shown 
that there is at most one imaginary quadratic field with $|D|>5460$ and $E(D)=2$. 
In \cite{BoKi}, \cite{We} it is shown (ineffectively) that there are finitely 
many imaginary quadratic fields with $E(D)=3$. In \cite{HB}, Theorem 2, it is 
observed that for given $r\geq 0$ there are finitely many imaginary quadratic 
fields with $E(D)=2^r $ or $E(D)=3\cdot 2^r$. (See also \cite{EK}.) 
Finally, there are finitely many imaginary quadratic fields with $E(D)=5$ (\cite{HB}, Theorem 1).

In this note we are interested to determine all $D$ such that the
class group has exponent at most 8, i.e. $\Cl(D)^\exponent$ is the trivial
group for some $\exponent \leq 8$. In other words we want that the class 
group is of type 
$C_2^r \times C_3^s, C_5^r, C_7^r$, or 
$C_2^r\times C_4^s \times C_8^t$, 
where $C_i$ denotes the cyclic group of order $i$. 
For any given $r,s$ it is known that there are infinitely many
$D<0$ such that $C_2^r\times C_4^s$ is a subgroup of $\Cl(D)$.  
%The $2$--rank of an abelian
%group $A$ is defined as the dimension of the $\F_2$--vector space
%$A/A^2$. Let $\omega(D)$ denote the number of distinct prime factors
%of $D$. It is known from genus theory that $\rk_2(\Cl(D))=\omega(D)-1$. 
%Therefore in the exponent 4 problem we have to consider class groups
%of type $C_2^r\times C_4^s$, where $r+s=\omega(D)-1$.  

Our computations show the following:
\begin{theorem}\label{th1}
There are exactly 1555 % 9+56+17+203+27+432+33+778 
imaginary quadratic fields with discriminant $|D| \leq 3.1 \cdot 10^{20}$ 
and class group of exponent $\leq 8$.

\begin{table}[h]
\begin{tabular}{c|c|c}
Exponent & Number of fields found & Field with largest discriminant \\
\hline
1        & 9                     & $\Q(\sqrt{-163})$               \\  
2        & 56                    & $\Q(\sqrt{-5460})$               \\
3        & 17                    & $\Q(\sqrt{-4027})$               \\
4        & 203                   & $\Q(\sqrt{-435435})$             \\       
5        & 27                    & $\Q(\sqrt{-37363})$              \\
6        & 432                   & $\Q(\sqrt{-5761140})$            \\
7        & 33                    & $\Q(\sqrt{-118843})$             \\
8        & 778                   & $\Q(\sqrt{-430950520})$
\end{tabular}

\caption{Number of imaginary quadratic fields with small exponent\label{tab_th_1}}
\end{table}
The discriminants with more than 7 decimal places are
-11148180, -12517428, -15337315, -15898740, -17168515, -28663635, -29493555, -31078723, -430950520.

\end{theorem}

\begin{theorem}
Assuming ERH, our computations found all fields for exponent up to 5 and 8. 
Assuming the non-existence of Siegel-zeros, 
our computation found all fields with exponents 2,4,8.

Without assuming any unproven statement we can conclude that 
there is at most one missing field with exponent 2,4 or 8. 
\end{theorem}

We remark that for exponent $2 \leq c\leq 8, c\ne 7$ it is (unconditionally) known that there
are only finitely many imaginary fields with exponent $c$. The explicit lists are only known using
suitable conjectures.

Let $$2=p_1<p_2<p_3<...$$ be the
sequence of prime numbers.  For $n\geq 1$
let $$d_n:=p_1\cdot\ldots\cdot p_n$$ be the product of the first $n$
prime numbers. In Section \ref{theoreticalbounds} we determine for
given $r\geq 1$ a number $N_{2^r}$ such that there is at most one
imaginary quadratic field with $|D|\geq d_{N_{2^r}}$ and $E(D)=2^r$.
For $r=1$ we have $ N_2=11 $ and $d_{N_2}\leq 2.01\cdot 10^{11}$.  For
$r=2$ we have $ N_{4}=24 $ and $d_{N_{4}}\leq 2.38\cdot 10^{34}$.  For
$r=3$ we have $ N_{8}=58 $ and $d_{N_{8}}\leq 3.17\cdot 10^{110}$.

\section{Some theoretical estimates}

As before let $D<0$ be a fundamental discriminant. 
% and $d$ be the squarefree
%part of $D$, i.e. $d=D$ if $D$ is odd, or $d=D/4$ if $D$ is even. 
%Denote by  $\sigma$ the non-trivial automorphism of $K$. 
In the following we want to use the knowledge of a small splitting prime $p$,
i.e. a prime $p\nmid D$ such that $p\OO_K=\idp \cdot \overline{\idp}$ splits into two
different prime ideals in $K$. We are interested to give some lower bound of the order
of the ideal $\idp$ in the class group $\Cl(D)$. The following lemma appears in different forms at least in \cite{De} p. 174-175, 
\cite{BoKi} Lemma 2, \cite{We} Lemma 5.  
\begin{lemma}\label{lem1}
  Let $c>0$ be an integer, let $p$ be a split prime in $K$, and let $\idp$ be a prime divisor of $p$ in $K$. If $p^\exponent < |D| /4$, then the order of $\idp$ in  the class group $\Cl(D)$ is strictly larger than $c$.
\end{lemma}

As shown in Lemma~\ref{lem1} a small split prime in an imaginary quadratic number field
already gives a good lower bound for the exponent of the class group. Here
we want to use the extended Riemann hypothesis (ERH). Using this, we can prove:
\begin{theorem}{\label{Bach_Bound}}
  Let $K$ be a quadratic number field such that the absolute value of
  the discriminant $D$ is larger than $e^{25} \approx 7.2\cdot 10^{10}$. 
  Assume the extended Riemann hypothesis. Then there exists a split
  prime $p$ such that
  $$p \leq (1.881\log(|D|)+2 \cdot 0.34 +5.5)^2.$$
\end{theorem}
\begin{proof}
  This is the result of Table 3 on page 1731 in \cite{BaSo}.
\end{proof}
We remark that in this paper there are similar statements with weaker constants for small
discriminants. We used those in the following table for exponent smaller or equal to 3.

\begin{table}[h]
\begin{tabular}{c|c}
Exponent & Bound for $|D|$ \\
\hline
1 & $1.7 \cdot 10^3$ \\
2 & $6 \cdot 10^6$ \\
3 & $9.7 \cdot 10^{10}$  \\
4 & $3.4 \cdot 10^{15}$ \\
5 & $2.3 \cdot 10^{20}$ \\
6 & $2.5 \cdot 10^{25}$ \\
7 & $3.9 \cdot 10^{30}$ \\
8 & $8.9 \cdot 10^{35}$
\end{tabular}
\smallskip 

\caption{ERH based bound (Theorem \ref{Bach_Bound}) for various exponents}
\end{table}

\section{Using Siegel-Tatuzawa bounds}

In the already cited paper by Weinberger \cite{We}, it is suggested
to use efficient bounds based on Siegel zeros. Compared to the original
Siegel bounds they have the advantage that the constants can be explicitly
computed. Weinberger used this approach to determine all (assuming that there
are no Siegel zeros) imaginary quadratic number fields of exponent 2. If we
do not assume any unproven conjecture it is shown that at most one field
is missing.

For a fundamental discriminant $D$ we define the character
$$\chi(n):=\chi_D(n):=\jac{D}{n},$$
where the symbol denotes the Kronecker symbol. We associate to it the following
$L$--series:
$$L(s,\chi)=\sum_{n=1}^\infty \frac{\chi(n)}{n^s}, \Re(s)>0.$$
Let $h(D):=|\Cl(D)|$ be the class number of the field $K=\Q(\sqrt{D})$. 
For $D<-4$ it is well known that
$$h(D)=\frac{\sqrt{|D|}L(1,\chi)}{\pi}.$$
We are interested in good lower bounds for the value $L(1,\chi)$.
It is known that there are no zeros of $L(s,\chi)$ for $\Re(s)\geq 1$.
We get efficient lower bounds, if we assume that there
are no real zeros of $L(s,\chi)$ which are close to 1. Tatuzawa \cite[Lemma 9 and Theorem 2]{Tat} 
proved:
 \begin{lemma}\label{Siegel1}   %
  Let $0<\eps<1/2$. There is at most one $|D|\geq \max(e^{1/\eps},e^{11.2})$ such that
  $$L(1,\chi)\leq \frac{0.655\eps}{|D|^\eps}.$$
    In this case $L(s,\chi)$ has a real zero $s$ with $1-\eps/4<s<1$.
\end{lemma}
As consequences we obtain
\begin{lemma}\label{Siegel} Let 
$|D|>e^{11.2}$. If $L(s,\chi)\neq 0$ for $1-\frac{1}{4\log |D|}\leq s<1$ then 
$$h(D)>\frac{0.655}{\pi e}\cdot\frac{\sqrt{|D|}}{\log |D|}.$$
\end{lemma}
\begin{proof}
We apply Lemma \ref{Siegel1} with $\eps=1/\log(|D|)$ and obtain 
$$L(1,\chi)> \frac{0.655\eps}{|D|^\eps}=\frac{0.655}{e\log |D|}.$$ Using the class number formula  
$L(1,\chi)=\frac{\pi h(D)}{\sqrt{|D|}}, $ we get the assertion.
%$$h(D)>\frac{0.655}{\pi e}\cdot\frac{\sqrt{|D|}}{\log |D|}.$$
\end{proof}

\begin{lemma}\label{Siegel_one}
Let $A\geq e^{11.2}$. For all $|D|\geq A $ with at most one exception it holds that 
$$h(D)> \frac{0.655}{\pi\cdot \log A}\cdot |D|^{\frac12-\frac{1}{\log A}}.$$
For $|D|=A^m$ with $m\geq 1$ we get:
$$h(D)> \frac{m \cdot 0.655\cdot  \sqrt{|D|}}{\pi\cdot e^m\cdot \log |D|}$$
with at most one exception.
\end{lemma}

\begin{proof}
We apply Lemma \ref{Siegel1} with  $\eps=1/\log A$ and obtain  
$$L(1,\chi)> \frac{0.655\eps}{|D|^\eps},$$ 
$$h(D)=\frac{\sqrt{|D|}L(1,\chi)}{\pi}> \frac{0.655}{\pi\cdot \log A}\cdot |D|^{\frac12-\frac{1}{\log A}}$$
for all $|D|\geq A $ with at most one exception. The second statement is a straightforward computation.
\end{proof}

The following lemma gives some improvement for the case with one exception.
\begin{lemma}\label{Chen}
  Let $A\ge 10^6$, $\eps:=1/\log A$ and $m:=\frac{\log |D|}{\log
    A}$. Then for all $D$ with $|D|\geq A$ we have with at most one
  exception:
  $$L(1,\chi) \geq \min \left(\frac{1}{7.732 \log{|D|}}, 1.5 \cdot 10^6 \frac{\eps}{|D|^\eps} \right)$$
  \begin{equation}\label{b_ch}
 h(D)\geq \min \left(\frac{\sqrt{|D|}}{\pi \cdot 7.732 \cdot \log |D|},
 \frac{m \cdot 1.5\cdot 10^6 \cdot \sqrt{|D|}}{\pi\cdot e^m \cdot \log{|D|}}\right).
 \end{equation}
  Note that for $m\leq 19.2$ the first number is the minimum.
\end{lemma}
\begin{proof}
  This is the main result of \cite{Ch}.
\end{proof}

Let us compare the result of Lemma \ref{Chen} with Lemmata \ref{Siegel} and \ref{Siegel_one}.
When comparing Lemma \ref{Siegel} with Lemma \ref{Chen}, then we see that Lemma \ref{Siegel} gives
a lower class group bound which is about a factor 2 better. But here we have to assume that there are
no Siegel zeros. When we compare Lemma \ref{Siegel_one} with Lemma \ref{Chen}, then we see that for $m\geq 2.6$
Lemma \ref{Chen} is better. Note that the second number in the minimum is always better 
than the bound in Lemma \ref{Siegel_one}. Furthermore it is important to note that 
$(10^6)^{19.2} >10^{115}$ which is sufficient for all our computations for exponent 8.

\section {Fields with exponent a power of two} \label{theoreticalbounds}

Weinberger proved in \cite{We}, Theorem 1,  
that there is at most one imaginary quadratic field with $|D|>5460$ and $E(D)=2$. 
We want to generalize his result to exponent $2^r$.
Let 
$$2=p_1<p_2<p_3<...$$ 
be the sequence of prime numbers. For $n\geq 1$ let $$d_n:=p_1\cdot\ldots\cdot p_n$$ be the 
product of the first $n$ prime numbers. In this section we determine for given $r\geq 1$ a 
number $N_{2^r}$ such that there is at most one imaginary quadratic field with 
$|D|\geq d_{N_{2^r}}$ and $E(D)=2^r$.

\begin{lemma}\label{lem6}
Let $D<0$ be a fundamental discriminant. If $E(D)=2^r$ with $r\geq1 $ then $$h(D)\leq 2^{r(\omega(D)-1)}.$$
\end{lemma}

\begin{proof}
   If $E(D)=2^r$ then the class group $\Cl(D)$ is isomorphic to 
   $C_2^{a_1}\times C_{2^2}^{a_2}\times\ldots C_{2^r}^{a_r}$, $a_1\geq 0,\ldots, a_{r-1}\geq 0$, $a_r>0$, 
   and we have $$h(D)=2^{a_1+2a_2+\ldots+ra_r}.$$
	By genus theory we have that 
	$$ a_1+\ldots+a_r=\omega(D)-1, $$ hence 
	$$h(D)=2^{a_1+2a_2+\ldots+ra_r}\leq 2^{r(a_1+\ldots+a_r)}=2^{r(\omega(D)-1)}. $$
	
\end{proof}

\begin{lemma}\label{lem7}
If $D$ is a fundamental discriminant then  
$$|D|\geq d_{\omega(D)}.$$
\end{lemma}

\begin{proof} We have that 
$$ D=\prod_{p|D}p^* $$ 
with 
$2^*\in\{-4, -8, 8\}$ and $p^*=(-1)^{\frac{p-1}{2}}p $ for $p\neq 2$, hence 
$$|D|\geq p_1\cdot\ldots\cdot p_{\omega(D)}= d_{\omega(D)}.$$
\end{proof}

\begin{theorem}\label{upp_bound}
Let $r\geq 1$ be an integer. Let $N_{2^r}$ be the smallest of the integers $N$ such that 
$$d_N\geq e^{11.2},$$
$$p_N^{\frac12-\frac{1}{\log d_N}}\geq 2^r,$$
and 
$$\frac{0.655}{\pi\cdot e}\cdot\frac{\sqrt{d_N}}{\log d_N  }\geq 2^{r(N-1)}.$$
Let $D<0$ be a fundamental discriminant with $|D|\geq d_{N_{2^r}}$. If $L(s,\chi)\neq 0$ for $1-\frac{1}{4\log |D| }\leq s<1$ then $E(D)\neq 2^r$. 

Without any assumption on zeros of $L$-functions, there is at most one $D$ with $|D|\geq d_{N_{2^r}}$ and $E(D)=2^r$. 
\end{theorem}

\begin{proof} %
Let $N\geq 1$ be an integer with the three properties from the hypothesis.
Assume that $L(s,\chi)\neq 0$ for $1-\frac{1}{4\log |D|}\leq s<1$. 
We apply Lemma \ref{Siegel} and obtain 
$$h(D)>\frac{0.655}{\pi e}\cdot\frac{\sqrt {|D|}}{\log |D| }.$$
Suppose that $E(D)= 2^r$. Lemma \ref{lem6} implies

$$  2^{r(\omega(D)-1)}>\frac{0.655}{\pi e}\cdot\frac{\sqrt{|D|}}{\log |D|}\geq \frac{0.655}{\pi e}\cdot\frac{\sqrt{d_N}}{\log d_N} \geq 2^{r(N-1)}, $$

hence 

$$\omega(D)> N. $$ By Lemma \ref{lem7} we have that   

$$ |D|\geq d_{\omega(D)}\geq  d_N\cdot p_N^{\omega(D)-N}, $$
so
$$ 2^{r(\omega(D)-1)}>\frac{0.655}{\pi e}\cdot\frac{\sqrt{|D|}}{\log|D|}
\geq \frac{0.655}{\pi e}\cdot\frac{\sqrt{d_N}\cdot \sqrt{p_N}^{\omega(D)-N}}
{\log d_N+(\omega(D)-N) \log p_N}=$$
$$=\frac{0.655}{\pi e}\cdot\frac{\sqrt{d_N}\cdot 2^{r(\omega(D)-N)}\cdot  (\frac{ \sqrt{p_N}}{2^r}) ^{\omega(D)-N}}
{\log d_N+(\omega(D)-N) \log p_N}, $$

$$  2^{r(N-1)}>\frac{0.655}{\pi e}\cdot\frac{\sqrt{d_N}\cdot  (\frac{ \sqrt{p_N}}{2^r}) ^{\omega(D)-N}}
{\log d_N+(\omega(D)-N) \log p_N}\geq \frac{0.655}{\pi e}\cdot\frac{\sqrt{d_N}}{\log d_N}, $$
since the hypothesis $p_N^{\frac12-\frac{1}{\log d_N}}\geq 2^r$ implies  $\sqrt{p_N}> 2^r$ so the function 
$$x\mapsto   \frac{(\frac{ \sqrt{p_N}}{2^r}) ^{x-N}}
{\log d_N+(x-N) \log p_N}, x\geq N$$
is increasing. This contradicts the hypothesis 

$$\frac{0.655}{\pi\cdot e}\cdot\frac{\sqrt{d_N}}{\log d_N  }\geq 2^{r(N-1)}.$$ So $E(D)\neq 2^r$.

We make now no assumption on zeros of $L$-functions. We apply Lemma 
\ref{Siegel_one} with $A=d_N$ and obtain  
 $$h(D)> \frac{0.655}{\pi\cdot \log d_N}\cdot |D|^{\frac12-\frac{1}{\log d_N}}$$
for all discriminants $D<0$ such that $|D|\geq d_N$ with at most one exception. Let $D<0$ be a discriminant such that 
$|D|\geq d_N$ and 
$$h(D)> \frac{0.655}{\pi\cdot \log d_N}\cdot |D|^{\frac12-\frac{1}{\log d_N}}.$$
Suppose that $E(D)=2^r$. By Lemma \ref{lem6} and the choice of $N$ we have that 
$$2^{r(\omega(D)-1)}\geq h(D)> \frac{0.655}{\pi\cdot \log d_N}\cdot |D|^{\frac12-\frac{1}{\log d_N}}\geq 2^{r(N-1)} $$
hence 
$$\omega(D)>N.$$
By Lemma \ref{lem7}  we have that  
$$|D| \geq d_{\omega(D)}>  d_N\cdot p_N^{\omega(D)-N}, $$
so
$$2^{r(\omega(D)-1)}\geq h(D)> \frac{0.655}{\pi\cdot \log d_N}\cdot |D|^{\frac12-\frac{1}{\log d_N}}> $$
$$>\frac{0.655}{\pi\cdot \log d_N}\cdot d_N^{\frac12-\frac{1}{\log d_N}}\cdot p_N^{(\omega(D)-N)\cdot (\frac12-\frac{1}{\log d_N})}=$$
$$=\frac{0.655}{\pi}\cdot \frac{ d_N^{\frac12-\frac{1}{\log d_N}}}{\log d_N}\cdot 
2^{r(\omega(D)-N)}\cdot \left(\frac{p_N^{\frac12-\frac{1}{\log d_N}}}{2^r} \right)^{\omega(D)-N} $$
$$\geq \frac{0.655}{\pi}\cdot \frac{ d_N^{\frac12-\frac{1}{\log d_N}}}{\log d_N}\cdot 2^{r(\omega(D)-N)},$$
since $$\left(\frac{p_N^{\frac12-\frac{1}{\log d_N}}}{2^r} \right)^{\omega(D)-N}\geq 1 $$ by the hypothesis $p_N^{\frac12-\frac{1}{\log d_N}}\geq 2^r$. It follows that 
$$2^{r(N-1)} >\frac{0.655}{\pi}\cdot \frac{ d_N^{\frac12-\frac{1}{\log d_N}}}{\log d_N}=
\frac{0.655}{\pi\cdot e}\cdot \frac{\sqrt{d_N}}{\log d_N  },$$
in contradiction with the hypothesis 
$$\frac{0.655}{\pi\cdot e}\cdot\frac{\sqrt{d_N}}{\log d_N  }\geq 2^{r(N-1)}.$$
So $E(D)\neq 2^r$.

\end{proof}

{\bf Example.} For $r=1$ we have $ N_2=11 $ and $d_{N_2}\leq 2.01\cdot 10^{11}$. For $r=2$ we have 
$ N_{4}=24 $ and $d_{N_{4}}\leq 2.38\cdot 10^{34}$. For $r=3$ we have 
$ N_{8}=58 $ and $d_{N_{8}}\leq 3.17\cdot 10^{110}$.

\section{Algorithm using Siegel bounds}

In this section we want to use the estimates from Lemmata \ref{Siegel}
and \ref{Chen} in order to compute all imaginary quadratic
number fields with exponent 4 and 8. When we
use the bounds from Lemma \ref{Siegel}, we potentially miss 
fields such that the corresponding $L$-series has a Siegel zero
close to 1. The estimate from Lemma \ref{Chen} is a little bit
weaker. This leads to more expensive computations, but it has the
advantage that we can prove that we miss at most one field.  These estimates are not valid for fields with
small discriminant. This is not a big problem, since the class groups
of fields with small discriminants are known. E.g. the web-page
\cite{BM}
provides a table of all quadratic fields up to absolute discriminant
$10^7$. We checked the small fields independently by computations in
Magma we do not describe here.

Let $\exponent \in\{4,8\}$ be the exponent we are looking for. Then we split
our problem by looking at discriminants with $k$ different prime
factors. The estimates in Section \ref{theoreticalbounds} show that we
can bound the maximal number of prime factors. In the following we
write our discriminants as a product of $k$ fundamental discriminants
$p^*$. % which are up to sign of prime power order. 
For every odd prime
$p$ it holds that $p^*=(-1)^{\frac{p-1}{2}}p$, and   $2^* \in
\{-4,-8,8\}$.

In the following let
$$D=p_1^* \cdots p_k^*, \mbox{ where we assume that }p_1<\cdots<p_k.$$
From genus theory it is well known that the 2-rank of the class group
$\Cl(D)$ is exactly $k-1$. When we assume that the class group is of
exponent $c$, then the maximal possible class group is
$(\Z/c\Z)^{k-1}$ and therefore of order $c^{k-1}$. Using the estimates
of Lemma \ref{Siegel} or \ref{Chen} we can compute an upper
bound for the absolute value of $D$. The basic idea of the algorithm
is to test all $D$ smaller than this bound consisting of exactly $k$
prime factors. Especially, when $k$ is large it is not practical to
list and test all those $D$. Therefore we would like to reduce this
list further. The $4$-rank of class groups of quadratic number fields
is well studied and there are nice formulas to compute it. Furthermore,
it is known that the average $4$-rank is small \cite{FoKl1, Ger}. Assume that
we know the $4$-rank $r_4$ of $\Cl(D)$ and denote by $r_2:=k-1$ the 2-rank.
Then the maximal possible class group of exponent $\exponent$ improves to
$$(\Z/c\Z)^{r_4} \times (\Z/2\Z)^{r_2-r_4} \mbox{ of order }2^{r_2-r_4}c^{r_4}.$$
This gives an improvement by factor $(c/2)^{r_2-r_4}$.

\subsection{Redei matrices}\label{redeimat}
In this section we want to study the $4$-rank of the class group of quadratic
number fields. These things are well known and based on the works of Redei.
In the following we use the Kronecker symbols $\jac{D}{p}$ and corresponding to
$D=p_1^* \cdots p_k^*$ we define the matrix $M=(c_{ij})\in\F_2^k$ via
$$(-1)^{c_{ij}}:=\jac{p_j^*}{p_i} \mbox{ for }1\leq j\ne i\leq k
\mbox{ and }c_{ii}:=\sum_{j=1,j\ne i}^k c_{ij}, 1\leq i \leq k.$$
\begin{theorem}[Redei]\label{redei}
  Let $D=p_1^* \cdots p_k^*$. Then $\rk_4(\Cl(D)) = k-1-\Rank(M).$
\end{theorem}
See \cite{Red} 10.b, \cite{Ger} (2.7).

The relation $c_{ii}:=\sum_{j=1,j\ne i}^k c_{ij}$ for $1\leq i \leq k$
shows that the last column of $M$ is the sum of the first $k-1$
columns and therefore dependent from the first $k-1$
columns. Furthermore we see by the multiplicativity of the Kronecker
symbol that $$(-1)^{c_{ii}} = \jac{\prod\limits_{j=1,j\ne i}^k p_j^*}{p_i}.$$
In order to understand this matrix, the following lemma is helpful. It deals
also with the prime 2 except when $2^*=-4$.
\begin{lemma}\label{pq}
  Let $p_i$ and $p_j$ different prime numbers such $p_i^*\ne -4 \ne p_j^*$. Then
  \begin{equation}\label{pqe} 
    \jac{p_i^*}{p_j}\jac{p_j^*}{p_i}=\begin{cases}1 & p_i^*>0 \mbox{ or }p_j^*>0\\ -1 & p_i^*<0 \mbox{ and }p_j^*<0. \end{cases}
  \end{equation}
\end{lemma}
The proof is straigthforward from the reciprocity law.
From this we see that if we have two different primes congruent to 
3 modulo 4 dividing $D$ then we get an entry 1
in our matrix. Furthermore the matrix is not symmetric.

In our algorithm we compute a lot of Redei matrices and the
corresponding ranks. In order to simplify these computations we look
when the last row is the sum of the first $k-1$ rows.
\begin{lemma}
  Let $D=p_1^*\cdots p_k^* <0$ be a product of fundamental
  discriminants with $p_1<p_2<\ldots<p_k$ and denote by $M$ the Redei
  matrix and by the vector $(d_1,\ldots,d_k)$ the sum of the rows of $M$.  

  Assume that $p_1^*\ne -4$. Then the vector $(d_1,\ldots,d_k)$ is zero.
  In the case $p_1^*=-4$ we define $\tilde{D} :=  -D / 4$ and we get
   $$\left((-1)^{d_1},\ldots,(-1)^{d_k} \right)
    = \left(\jac{2}{\tilde{D}},\jac{2}{p_2},\ldots,\jac{2}{p_k}\right)\, .$$
  \end{lemma}
\begin{proof}
The sum of the
$j$-th column is 0 if and only if the product of the corresponding symbols in the exponents are 1.
We get for this product:
\begin{equation}\label{eq1}
\prod_{i=1,i\ne j}^k \jac{p_j^*}{p_i} \cdot \jac{\prod\limits_{i=1,i\ne j}^k p_i^*}{p_j}
=\prod_{i=1,i\ne j}^k \jac{p_j^*}{p_i}\jac{p_i^*}{p_j} .
\end{equation}
Let us assume that $p_1^*\ne -4$. Therefore we are able to apply Lemma \ref{pq}. 
If $p_j^*>0$ then all factors are 1.
If $p_j^*<0$ then the number of negative $p_i^*\ne p_j^*$ is even and therefore the product is 1.

It remains to study the case $p_1^*=-4$. If we multiply the product \eqref{eq1} with 
$\prod_{i=2}^k\jac{2}{p_i}$ if $j=1$ and with $\jac{2}{p_j}$ for $j\ne 1$ we get 
the product in the situation that $2^*=-8$ and we know that this product
is 1. Therefore the product is like expected.
\end{proof}
Note that in the case $p_1^*=-4$ we only get row sum 0, if all odd prime divisors are 
congruent to $\pm 1 \bmod 8$.

Using Theorem \ref{redei} we know that $\rk_4(\Cl(D)) = k-1-\Rank(M).$ Therefore we are interested to get
good lower bounds for the rank of the Redei matrix $M$. Denote by $t$ the number of negative $p_i^*$. Then
we get $\Rank(M) \geq (t-1)/2$, see \cite{Sue}. This paper also discusses the cases where this bound is sharp.

\subsection{Using the Redei matrix} 

In this section we describe an algorithm to compute all fields with exponent $\exponent =2^r$, where we focus on the
cases $c=4,8$. In order to use the Redei matrices we restrict to discriminants with
exactly $k$ prime factors. Using Lemma \ref{Chen} we can give an upper bound for the number of prime factors
for fields of exponent $2^r$ which is valid for all fields with at most one exception. In the example after
Theorem \ref{upp_bound} we get 11, 24, and 58 for the exponents 2,4,8, respectively. We remark that these bounds
can be improved, but the following algorithm is very efficient for the cases close to the bound. For a given exponent
$\exponent =2^r$ we call the following algorithm for all $k\geq 1$ up to the computed upper bound.

For the algorithm we have to make the decision if we want to use the
lower bound of Lemma \ref{Siegel} or of Lemma \ref{Chen}. The first lemma has the advantage that the bound is better
and therefore the computation will be faster. In this case we only compute all wanted fields which have no
Siegel zero. If we assume that no Siegel zeros exist, then this computation is complete. The bound of the second lemma
is weaker, but it has the advantage that we miss at most one field (if it exists, it has a Siegel zero). In our range
for exponent 8, this bound is about a factor 2 weaker than the first bound. In order to simplify the presentation we
only give the description of the algorithm using the bound of Lemma \ref{Chen}. 

The main algorithm to call is Algorithm \ref{algo}. This algorithm computes the global variables $B_0,\ldots,B_k$ which will
be used in Algorithms \ref{Check} and \ref{NextTuple}. The main idea of the following algorithms is that the knowledge of
a factor of $D$ gives some partial information on the Redei matrix. This information can be used to give an upper bound
on the $4$-rank of the class group which then gives improved bounds on the maximal possible discriminant. In theory we expect
that the average $4$-rank is close to 1~\cite{Ger} and therefore the upper bound $c^{k-1}$ for the class number of a field with exponent $\exponent$
is quite pessimistic. The described approach improves the upper bound of the class number, and therefore the maximal possible
discriminant, when the lower rank bound of the Redei matrix increases, and therefore the possible $4$-rank decreases.

The goal of the first algorithm is to give a quick check, if the
fundamental discriminant $D=p_1^*\cdots p_k^*$ has an exponent which
is a divisor of $\exponent$. The correctness of the algorithm is obvious by
using Lemma \ref{lem1}. If possible, we try to avoid the actual
computation of the class group.

\Algo{Check}{Check$(\exponent,p_1^*,\ldots,p_k^*)$}
{Exponent $\exponent =2^r$, prime fundamental discriminants $p_1^*,\ldots,p_k^*$.}
{Return true, iff $D=p_1^*\cdots p_k^*<0$ and the exponent of the class group of $\Q(\sqrt{D})$ is a divisor of $\exponent$.}
{
\item[Step 1] If $D:=p_1^*\cdots p_k^*>0$ then return false.
\item[Step 2] Compute the smallest prime $q$ which is split in $K:=\Q(\sqrt{D})$.
\item[Step 3] If $q^\exponent<|D|/4$ then return false (see Lemma \ref{lem1}).
\item[Step 4] If the $c$-th power of a prime ideal above $q$ is not principal, then return false. This test can be done
  most efficiently  by using binary quadratic forms.
\item[Step 5] Repeat the test of Step 4 with the 2nd smallest splitting prime.
\item[Step 6] Compute the rank $s$ of the Redei matrix of $D$. If $|D| > B_s$ then return false ($B_s$ is a global variable computed in Algorithm \ref{algo}).
\item[Step 7] Compute the class group of $K$. If the exponent divides $\exponent$ then return true, otherwise return false. 
}

  Let $D=p_1^*\cdots p_k^*$ be a negative fundamental discriminant and assume that $p_1^*,\ldots,p_\ell^*$ are known to us.
  Let $M$ be the Redei matrix of $D$ defined in Section \ref{redeimat}. Denote by $N$ the minor defined by the first
  $\ell$ rows and first $\ell$ columns. Trivially, we get that $\Rank(N) \leq \Rank(M)$. Since $p_1^*,\ldots,p_\ell^*$ are known to us,
  we can compute all entries of $N$ except the diagonal. The following algorithm tests all possible combinations for the diagonal
  and therefore computes a lower bound for the rank of $N$ and $M$.

\Algo{lowerredei}{LowerRedeiBound$(p_1^*,\ldots,p_\ell^*)$}
{Prime fundamental discriminants $p_1^*,\ldots,p_\ell^*$ with $\ell<k$.}  
{Returns a lower bound for the rank of the Redei matrix of all $D$ where $D$ has exactly
  $k$ prime factors and $p_1^*\cdots p_\ell^* \mid D$.}
{
\item[Step 1] For all $1\leq i\ne j \leq \ell$ compute $c_{ij}$ via $(-1)^{c_{ij}}:=\jac{p_j^*}{p_i}$.
\item[Step 2] For all $(a_1,\ldots,a_\ell)\in \F_2^\ell$ compute the rank of the matrix $D=(d_{ij})$,
  where $d_{ij}=c_{ij}$ for $i\ne j$ and $d_{ii}=a_i$.
\item[Step 3] Return the minimal rank computed in Step 2.
  }

  In the following algorithm we denote by $P[i]$ the $i$-th prime number. We update the lower bound of the rank of the Redei matrix
  and append the next fundamental prime discriminant. 
  This function will call itself recursively. 
  The number $m$ is the index of the smallest prime that can be used next. % largest prime number used so far. 
%  When we look for odd discriminants $m=1$ to avoid the prime 2 (see also Algorithm \ref{algo}).
  
\Algo{NextTuple}{NextTuple(m,Discs,k,c)} {Number $m$ of the next prime to use, 
  list {\rm Discs}$= [p_1^*,\ldots,p_\ell^*]$ of fundamental
  prime discs, $k$, exponent $\exponent =2^r$}
{List of all discriminants $D$ with $k$ prime factors such
  that $p_1^*\cdots p_\ell^* \mid D$ (with at most one exception)}
{ \item[Step 1] $s:={\rm LowerRedeiBound}(p_1^*,\ldots,p_\ell^*)$.
     \item[Step 2] $B:=B_s$, {\rm res}:=[] (empty list), ${\rm bound}:=B/(|p_1^*\cdots p_\ell^*|)$, 
                   $i:=m$, $C:=P[i]\cdots P[i+k-\ell-1]$.
     \item[Step 3] While $C \leq {\rm bound}$ do
        \begin{enumerate}
        \item $p_{\ell+1}^* := (-1)^{(P[i]-1)/2} P[i]$.
        \item If $k=\ell+1$ then call Check$(c,p_1^*,\ldots,p_k^*)$
          and append $D=p_1^*\cdots p_k^*$ to {\rm res}, if the check is successful.
        \item If $k>\ell+1$ then call NextTuple(i,$\{p_1^*,\ldots,p_{\ell+1}^*\}$,k,c) and append the computed $D$`s to {\rm res}.
        \item $i:=i+1$, $C:=P[i]\cdots P[i+k-\ell-1]$.
        \end{enumerate}
      \item [Step 4] Return {\rm res}.
}

In the main algorithm we compute the global variables $B_0,\ldots,B_{k-1}$ and we split the computation into
four parts, depending on the behavior at 2.

\Algo{algo}{Computation of fields with exponent $2^r$}{Exponent $2^r$, number of prime factors $k$}{All fields with
  exponent $\exponent =2^r$ with at most one exception}{
\item[Step 1] By numerical approximation compute a bound $B_\ell$ for all $0\leq \ell \leq k-1$ such that
  for all $|D|>B_\ell$ we get that the minimum of \eqref{b_ch} is greater than $2^\ell \cdot c^{k-1-\ell}$.
\item[Step 2]  Call {\rm res1:=NextTuple}(2,\{1\},k,c).
\item[Step 3]  Call {\rm res2:=NextTuple}(2,\{-4\},k,c).
\item[Step 4]  Call {\rm res3:=NextTuple}(2,\{-8\},k,c).
\item[Step 5]  Call {\rm res4:=NextTuple}(2,\{8\},k,c).
\item[Step 6] Return the discriminants from {\rm res1, res2, res3, res4}.
  }

  We remark that there are obvious improvements in the implementation which we have not described here for simplification. 
  Note
  that there might be one missing example for the overall algorithm by using Lemma \ref{Chen}. 
  The reason is that we compute all fields with exponent dividing $\exponent$ assuming the bounds 
  of Lemma \ref{Chen}. Therefore missing examples have the property that the bounds
  of Lemma \ref{Chen} are wrong and this can happen at most one time.

  The algorithm described above only computes discriminants $D$ such that $|D|>10^6$. We could easily reduce this lower bound,
  but this lower bound is not an issue since there are known tables of class groups 
  for all quadratic fields of small discriminant. The web-page~\cite{BM}
  gives all fields up to $10^7$.

  The overall running time of our algorithm on one core is about 17
  hours when we use the estimates from Lemma \ref{Siegel_one}. The
  cases $k\in\{30,\ldots,58\}$ take about 50 seconds.  Only for small
  $k$ we need to compute class groups. The cases $k\leq 6$ take about
  half an hour. The most expensive cases are $k\in\{8,\ldots,12\}$
  which take more one hour each, the worst case being $k=10$ which
  takes almost 3 hours.

  When we use the estiamtes from Lemma \ref{Chen} which are sufficient to prove
  that we miss at most one example, then the running time is about 60 hours.
  In case we can take a bound which is a factor 2 better than the bound in Lemma
  \ref{Siegel_one}, the running time improves to less than 3 hours.

\section{Direct searching for small discriminants}

\begin{lemma}
For a prime $p$ and an exponent $\exponent$ there are less than $2p^{\exponent/2}$ imaginary quadratic fields 
$K =\Q(\sqrt{D})$, such that
$p$ splits in $K$ and the ideal above $p$ has order divisible by $\exponent$ in the class group.
\end{lemma}

\begin{proof}
The assumption implies that the equation $4 p^\exponent = x^2 + |D| y^2$ has an integer solution 
with $x \cdot y \not= 0$. 
Thus, the field $K$ is one of $\Q(\sqrt{-(4 p^\exponent - x^2)})$ for $x \in \Z$ and $4 p^\exponent - x^2 > 0$. 
This shows $1 \leq x < 2 p^{\exponent/2}$. Thus, there are less than $2p^{\exponent/2}$ fields.
\end{proof}

\begin{remarks}
\begin{enumerate}
\item
The lemma above results in an algorithm to enumerate all imaginary quadratic 
fields such that $p$ is a split prime and the exponent of the class group is 
a divisor of $\exponent$. It's complexity is $O(p^{\exponent/2})$.
More precisely, we first compute a finite list of fields that is a superset of the fields we are searching. 
The superfluous fields can easily be removed in a second step. 
\item
As a slight variation, we can enumerate all quadratic fields such that $p$ is 
the smallest split prime and the exponent of the class group is a divisor of $\exponent$. 
For this we just have to sift out all those fields that have a smaller
split prime. Note that this can be done without factoring $D y^2 = 4 p^\exponent - x^2$.
\item
A C-implementation of this approach lists all the fields with smallest split prime $\leq 197$ 
and $\exponent=8$ in less than 4 minutes. We find 268 fields with exponent 1,2,4 and 778 fields with exponent 8. 
The largest one is $\Q(\sqrt{-430950520})$.
\item
Doing the same computation with exponent 3, 5, 6, 7 takes less than a minute. 
%We find the following:
%\smallskip
%\begin{tabular}{c|c|c}
%Exponent & Number of field found & Largest discriminant \\
%\hline
%2        & 56                    & $\Q(\sqrt{-5460})$               \\
%3        & 17                    & $\Q(\sqrt{-4027})$               \\
%4        & 203                   & $\Q(\sqrt{-435435})$             \\       
%5        & 27                    & $\Q(\sqrt{-37363})$              \\
%6        & 432                   & $\Q(\sqrt{-5761140})$            \\
%7        & 33                    & $\Q(\sqrt{-118843})$             \\
%8        & 778                   & $\Q(\sqrt{-107737630})$
%\end{tabular}
%\smallskip
%
\item
An imaginary quadratic field with class group exponent 
$\leq 8$ not listed has smallest split prime $> 193$.
\item
Searching for imaginary quadratic fields with $|D|$ up to a given bound and smallest split 
prime $> 193$ can be done by a multiply-focused enumeration similar to~\cite{WW}.
I.e. we have to sift out all those $D$ that have a small split prime.
\end{enumerate}
\end{remarks}

\begin{remarks}
We want to search for fields with no split prime $p \leq 193$ and $|D| < 3 \cdot 10^{20}$.
This can be done by sieving. To maximize the speed we have to use bit-level operations 
and tables of pre-computed data. Further, non negative integers $< 2^{64}$ have the 
fastest arithmetic. Thus, the main loop should be restricted to this.

This results in the following approach:
\begin{itemize}
\item
We want to search for all imaginary quadratic fields $\Q(\sqrt{D})$ with smallest split prime $> 193$.
This implies $\left(\frac{D}{p} \right) \not= 1$ for all primes $p \leq 193$. 
\item
Using the chinese remainder theorem with the module 
$m = 16 \cdot 3 \cdot 5 \cdots 47 \approx 4.9 \cdot 10^{18}$
the number of possible  residue classes for $|D|$ is
$220172127436800 \approx 2.2 \cdot 10^{14}$.
\item
Let $r$ be the smallest non-negative representative of a feasible residue class mod $m$. 
We have to test the fields with $|D| = r, r + m,r + 2m,\ldots, r + 63 m$ in parallel.
\item
The prime 53 rules out a field if $\left(\frac{D}{53} \right) = 1$. 
For $|D| = r, r + m,r + 2m,\ldots, r + 63 m$, we can encode this in a sequence of 64 bits.
The $k$-th bit is 1 iff the field for $|D| = r + k m$ is not ruled out by 
$\left(\frac{-r-km}{53} \right) = 1$.

We tabulate these bit-sequences for each residue class of $r  \bmod 53$.
\item
Similarly each other prime $p \leq 193$ and each residue class $r \bmod p$ 
we get a sequence of 64 bits. The $k$-th bis is 1 iff the field for 
$|D| = r + k m$ is not ruled out by $\left(\frac{-r- km}{p} \right) = 1$.

This gives us further tables of bit-sequences. One table for each prime $p$ with 
one entry of 64 bits for each residue class of $r \bmod p$.
\item
To combine the information modulo the various primes, we have to pick those $|D|$, 
that are not ruled out by one of the primes up to 193. 
In the language of bit sequences this means that we have to do logical {\tt and} of the sequences.
\item
If a result bit of the {\tt and} is 0, the corresponding field is ruled out.
\item
If a result bit is $1$, the field $\Q(\sqrt{-r - km})$ needs a more detailed inspection. 
\end{itemize}
\end{remarks}

This approach leads to the following algorithm.

%\begin{algorithm} [Multifocused bit-vector sieve] %\begin{enumerate}
\Algo{multi_sieve}{Multifocused bit-vector sieve}
{No input.}  
{Print all imaginary quadratic fields $\Q(\sqrt{D})$ without smallest split prime $p > 193$,
$|D| < 3.1 \cdot 10^{20}$ and $4p^8 > |D|$.}
{
\item[Step 1] Set the modules $m_1 = 3 \cdot 5 \cdot 7 \cdot 11 \cdot 13 \cdot 17 = 255255$, 
$m_2 = 19 \cdot 23 \cdot 29 \cdot 31 = 392863$, 
$m_3 = 37 \cdot 41 \cdot 43 \cdot 47  = 3065857$, and 
$m = 16 m_1 m_2 m_3 \approx 4.9 \cdot 10^{18}$.
\item[Step 2] For each module $m_i$ compute a list of the integers $r$ in $\{0..m_i-1\}$ 
such that $\left(\frac{-r}{p}\right) \not= 1$
for all primes $p$ dividing $m_i$.
\item[Step 3] For each prime $p \leq 193$ not dividing $m$ set up a list $(l_p[0],\ldots,l_p[p-1])$ 
of $p$ bit-vectors of length 64. 
The $k$-th bit in $l_p[i]$ is $0$ if and only if $\left(\frac{-(i + km)}{p} \right) = 1$.
\item[Step 4] In a quadruple loop run over the cartesian product of the 3 lists computed in step 2 and 
$\{3, 4, 8, 11 \bmod 16 \}$ and do the following:
\begin{enumerate}
\item[a)] 
Use the chinese remainder theorem to find the unique integer $0 \leq r < m$ congruent to the given
residues modulo $m_1, m_2, m_3, 16$. 
\item[b)] Do a logical {\tt and} of the bit-vectors $l_p[r \bmod p]$
for all prime $p$ with $53 \leq p \leq 193$.
\item[c)] If the $k$-th bit of the resulting bit-vector is $1$, 
the field $\Q(\sqrt{D})$ with $D:= -r - k \cdot m$ is suspicious.
\item[d)] Compute the smallest split prime $p$ for each suspicious field $\Q(\sqrt{D})$.
If $4p^8 > |D|$ then print the field.
\end{enumerate}
}
% \end{Algo}
%\end{enumerate}
%\end{algorithm}
 
\begin{remark}
The algorithm above finds all the imaginary quadratic fields $\Q(\sqrt{D})$ with  $|D| < 3.14 \cdot 10^{20}$,
no split-prime $\leq 193$. We print out only those fields that may have a class group exponent $\leq 8$.
The run time on a single core on an Intel i5 processor is about 40 days. However, we can the loop over the 
cartesian product in parallel. 
The result is as follows:
\begin{itemize}
\item
The algorithm above results in 1002279 imaginary quadratic fields.
\item
None of them has a class group of exponent $\leq 100$.

\end{itemize}
\end{remark}

%\newpage

%\bibliographystyle{plain}
%\bibliography{class4}
 
\end{document}